\documentclass{amsart}
\usepackage{amsmath}
\usepackage{amssymb}
\usepackage{amsfonts}

\newtheorem{theorem}{Theorem}
\theoremstyle{plain}

\newtheorem{corollary}{Corollary}

\newtheorem{definition}{Definition}

\newtheorem{lemma}{Lemma}

\newtheorem{proposition}{Proposition}
\newtheorem{remark}{Remark}

\numberwithin{equation}{section}
\input{tcilatex}

\begin{document}
\title[New Hadamard-Type Inequalities]{THE HADAMARD TYPE INEQUALITIES FOR $m-
$CONVEX FUNCTIONS}
\author{\c{C}etin Y\i ld\i z$^{\blacksquare ,\clubsuit }$}
\address{$^{\blacksquare }$ATAT\"{U}RK UNIVERSITY, K. K. EDUCATION FACULTY,
DEPARTMENT OF MATHEMATICS, 25240, CAMPUS, ERZURUM, TURKEY}
\email{yildizcetiin@yahoo.com}
\author{Mustafa G\"{u}rb\"{u}z$^{\bigstar }$}
\address{$^{\bigstar }$GRADUATE SCHOOL OF NATURAL AND APPLIED SCIENCES, A%
\u{G}RI \.{I}BRAH\.{I}M \c{C}E\c{C}EN UNIVERSITY, A\u{G}RI, TURKEY}
\email{mgurbuz@agri.edu.tr}
\author{Ahmet Ocak Akdemir$^{\spadesuit }$}
\address{$^{\spadesuit }$A\u{G}RI \.{I}BRAH\.{I}M \c{C}E\c{C}EN UN\.{I}%
VERSITY, FACULTY OF SCIENCE AND ARTS, DEPARTMENT OF MATHEMATICS, 04100,
CAMPUS, A\u{G}RI, TURKEY}
\email{ahmetakdemir@agri.edu.tr}
\date{June 04, 2010}
\keywords{Hermite-Hadamard's inequality, $m-$Convex, H\"{o}lder inequality,
Power mean inequality\\
$^{\clubsuit }$corresponding author}

\begin{abstract}
In this paper we obtained some new Hadamard-Type inequalities for functions
whose derivatives absolute values $m-$convex. Some applications to special
means of real\ numbers are given.
\end{abstract}

\maketitle

\section{INTRODUCTION}

Let $f:I\subset 
\mathbb{R}
\rightarrow 
\mathbb{R}
$ be a convex function defined on the interval $I$ of real numbers and a,b$%
\in I$, with $a<b$. The following inequality is well known as the
Hermite-Hadamard inequality for convex functions,%
\begin{equation*}
f\left( \frac{a+b}{2}\right) \leq \frac{1}{b-a}\int_{a}^{b}f(x)dx\leq \frac{%
f(a)+f(b)}{2}.
\end{equation*}

In recent years many authors have established several inequalities connected
to Hermite-Hadamard's\textit{\ }inequality. For recent results, refinements,
counterparts, generalizations and new\textit{\ }Hadamard-Type inequalities
see \cite{SSD}, \cite{USK} and \cite{YHT}.

A function $f:I\rightarrow 
\mathbb{R}
$ is said to be convex function on $I$ if the inequality%
\begin{equation*}
f(tx+(1-t)y)\leq tf(x)+(1-t)f(y),
\end{equation*}%
holds for all $x,y\in I$ and $t\in \lbrack 0,1].$

In \cite{TODER}, G. Toader defined $m-$convexity as the following:

\begin{definition}
The function $f:[0,b]\rightarrow 
\mathbb{R}
,$ $b>0$, is said to be $m-$convex where $m\in \lbrack 0,1],$ if we have%
\begin{equation*}
f(tx+m(1-t)y)\leq tf(x)+m(1-t)f(y)
\end{equation*}%
for all $x,y\in \lbrack 0,b]$ and $t\in \lbrack 0,1].$ We say that f is
m-concave if -f is m-convex.
\end{definition}

For recent results related to above definitions we refer interest of readers
to \cite{mkb},\cite{drgtd},\cite{ssdp}.

The following theorems which was obtained by Kavurmac\i\ et al. contains the
Hadamard-Type integral inequalities in \cite{HME}.

\begin{theorem}
Let $f:I\subset 
\mathbb{R}
\rightarrow 
\mathbb{R}
$ be a differentiable mapping on $I^{o}$ such that $f^{\prime }\in L[a,b],$
where $a,b\in I$ with $a<b.$ If $\left\vert f^{\prime }\right\vert $ is
convex function on $[a,b],$ then the following inequality holds:%
\begin{equation}
\left\vert \frac{f(a)+f(b)}{2}-\frac{1}{b-a}\dint\limits_{a}^{b}f(u)du\right%
\vert \leq \frac{b-a}{12}\left[ \left\vert f^{\prime }\left( \frac{a+b}{2}%
\right) \right\vert +\left\vert f^{\prime }\left( a\right) \right\vert
+\left\vert f^{\prime }\left( b\right) \right\vert \right] .  \label{1}
\end{equation}
\end{theorem}

\begin{theorem}
Let $f:I\subset 
\mathbb{R}
\rightarrow 
\mathbb{R}
$ be a differentiable mapping on $I^{o}$ such that $f^{\prime }\in L[a,b],$
where $a,b\in I$ with $a<b.$ If $\left\vert f^{\prime }\right\vert ^{q}$ is
convex function on $[a,b],$ for some fixed $q\geq 1,$ then the following
inequality holds:%
\begin{eqnarray}
&&\left\vert \frac{f(a)+f(b)}{2}-\frac{1}{b-a}\dint\limits_{a}^{b}f(u)du%
\right\vert  \label{2} \\
&\leq &\frac{b-a}{8}\left[ \left( \frac{\left\vert f^{\prime }\left( \frac{%
a+b}{2}\right) \right\vert ^{q}+2\left\vert f^{\prime }\left( a\right)
\right\vert ^{q}}{3}\right) ^{\frac{1}{q}}+\left( \frac{\left\vert f^{\prime
}\left( \frac{a+b}{2}\right) \right\vert ^{q}+2\left\vert f^{\prime
}(b)\right\vert ^{q}}{3}\right) ^{\frac{1}{q}}\right] .  \notag
\end{eqnarray}
\end{theorem}

The main purpose of this paper is to establish refinements inequalities of
right-hand side of Hadamard's type for $m-$convex functions.

\section{MAIN RESULTS}

In \cite{HME}, in order to prove some inequalities related to
Hermite-Hadamard's inequality Kavurmac\i\ et al.\ used the following lemma.

\begin{lemma}
Let $f:I\subset 
\mathbb{R}
\rightarrow 
\mathbb{R}
$ be a differentiable mapping on $I^{o}$ where $a,b\in I$ with $a<b.$ If $%
f^{\prime }\in L[a,b],$ then the following equality holds:%
\begin{eqnarray*}
&&\frac{(b-x)f(b)+(x-a)f(a)}{b-a}-\frac{1}{b-a}\overset{b}{\underset{a}{\int 
}}f(u)du \\
&=&\frac{(x-a)^{2}}{b-a}\underset{a}{\overset{b}{\int }}(1-t)f^{\prime
}(tx+(1-t)a)dt+\frac{(b-x)^{2}}{b-a}\overset{b}{\underset{a}{\int }}%
(t-1)f^{\prime }(tx+(1-t)b)dt.
\end{eqnarray*}
\end{lemma}

\begin{theorem}
Let $f:I\subset 
\mathbb{R}
\rightarrow 
\mathbb{R}
$ be a differentiable mapping on $I^{o}$ where $a,b\in I$ with $a<b.$ If $%
\left\vert f^{\prime }\right\vert $ is $m-$convex function on $[a,b]$ for
some fixed $m\in (0,1],$ then the following inequality holds:%
\begin{eqnarray*}
&&\left\vert \frac{(b-x)f(b)+(x-a)f(a)}{b-a}-\frac{1}{b-a}%
\dint\limits_{a}^{b}f(u)du\right\vert \\
&\leq &\frac{(x-a)^{2}}{b-a}\left[ \frac{\left\vert f^{\prime
}(x)\right\vert +2m\left\vert f^{\prime }(\frac{a}{m})\right\vert }{6}\right]
+\frac{(b-x)^{2}}{b-a}\left[ \frac{\left\vert f^{\prime }(x)\right\vert
+2m\left\vert f^{\prime }(\frac{b}{m})\right\vert }{6}\right]
\end{eqnarray*}%
where $m\in (0,1].$
\end{theorem}

\begin{proof}
From Lemma 1 and using property of absolute value, we get;%
\begin{eqnarray*}
&&\left\vert \frac{(b-x)f(b)+(x-a)f(a)}{b-a}-\frac{1}{b-a}%
\dint\limits_{a}^{b}f(u)du\right\vert \\
&\leq &\frac{(x-a)^{2}}{b-a}\dint\limits_{0}^{1}(1-t)\left\vert f^{\prime
}(tx+(1-t)a)\right\vert dt+\frac{(b-x)^{2}}{b-a}\dint\limits_{0}^{1}(1-t)%
\left\vert f^{\prime }(tx+(1-t)b)\right\vert dt.
\end{eqnarray*}%
Since $\left\vert f^{\prime }\right\vert $ is $m-$convex we can write;%
\begin{eqnarray*}
&&\left\vert \frac{(b-x)f(b)+(x-a)f(a)}{b-a}-\frac{1}{b-a}%
\dint\limits_{a}^{b}f(u)du\right\vert \\
&\leq &\frac{(x-a)^{2}}{b-a}\dint\limits_{0}^{1}(1-t)\left[ t\left\vert
f^{\prime }(x)\right\vert +m(1-t)\left\vert f^{\prime }(\frac{a}{m}%
)\right\vert \right] dt \\
&&+\frac{(b-x)^{2}}{b-a}\dint\limits_{0}^{1}(1-t)\left[ t\left\vert
f^{\prime }(x)\right\vert +m(1-t)\left\vert f^{\prime }(\frac{b}{m}%
)\right\vert \right] dt \\
&=&\frac{(x-a)^{2}}{b-a}\left[ \left\vert f^{\prime }(x)\right\vert
\dint\limits_{0}^{1}(t-t^{2})dt+m\left\vert f^{\prime }(\frac{a}{m}%
)\right\vert \dint\limits_{0}^{1}(1-t)^{2}dt\right] \\
&&+\frac{(b-x)^{2}}{b-a}\left[ \left\vert f^{\prime }(x)\right\vert
\dint\limits_{0}^{1}(t-t^{2})dt+m\left\vert f^{\prime }(\frac{b}{m}%
)\right\vert \dint\limits_{0}^{1}(1-t)^{2}dt\right] \\
&=&\frac{(x-a)^{2}}{b-a}\left[ \frac{\left\vert f^{\prime }(x)\right\vert
+2m\left\vert f^{\prime }(\frac{a}{m})\right\vert }{6}\right] +\frac{%
(b-x)^{2}}{b-a}\left[ \frac{\left\vert f^{\prime }(x)\right\vert
+2m\left\vert f^{\prime }(\frac{b}{m})\right\vert }{6}\right] .
\end{eqnarray*}%
This completes the proof.
\end{proof}

\begin{corollary}
In Theorem 3, if we choose $x=\frac{a+b}{2},$ we have%
\begin{equation*}
\left\vert \frac{f(a)+f(b)}{2}-\frac{1}{b-a}\dint\limits_{a}^{b}f(u)du\right%
\vert \leq \frac{b-a}{12}\left[ \left\vert f^{\prime }\left( \frac{a+b}{2}%
\right) \right\vert +m\left\vert f^{\prime }\left( \frac{a}{m}\right)
\right\vert +m\left\vert f^{\prime }\left( \frac{b}{m}\right) \right\vert %
\right] .
\end{equation*}
\end{corollary}

\begin{remark}
In Corollary 1, if we choose $m=1,$ (\ref{1}) inequality is obtained.
\end{remark}

\begin{theorem}
Let $f:I\subset 
\mathbb{R}
\rightarrow 
\mathbb{R}
$ be a differentiable mapping on $I^{o}$ where $a,b\in I$ with $a<b.$ If $%
\left\vert f^{\prime }\right\vert ^{q}$ is $m-$convex function on $[a,b],$
then the following inequality holds:%
\begin{eqnarray*}
&&\left\vert \frac{(b-x)f(b)+(x-a)f(a)}{b-a}-\frac{1}{b-a}%
\dint\limits_{a}^{b}f(u)du\right\vert \\
&\leq &\left( \frac{1}{p+1}\right) ^{\frac{1}{p}}\left[ \frac{(x-a)^{2}}{b-a}%
\left( \frac{\left\vert f^{\prime }(x)\right\vert ^{q}+m\left\vert f^{\prime
}(\frac{a}{m})\right\vert ^{q}}{2}\right) ^{\frac{1}{q}}\right. \\
&&\left. +\frac{(b-x)^{2}}{b-a}\left( \frac{\left\vert f^{\prime
}(x)\right\vert ^{q}+m\left\vert f^{\prime }(\frac{b}{m})\right\vert ^{q}}{2}%
\right) ^{\frac{1}{q}}\right] \text{ \ }
\end{eqnarray*}%
where $m\in (0,1].$
\end{theorem}

\begin{proof}
From Lemma 1, we can write;%
\begin{eqnarray*}
&&\left\vert \frac{(b-x)f(b)+(x-a)f(a)}{b-a}-\frac{1}{b-a}%
\dint\limits_{a}^{b}f(u)du\right\vert \\
&\leq &\frac{(x-a)^{2}}{b-a}\dint\limits_{0}^{1}(1-t)\left\vert f^{\prime
}(tx+(1-t)a)\right\vert dt \\
&&+\frac{(b-x)^{2}}{b-a}\dint\limits_{0}^{1}(1-t)\left\vert f^{\prime
}(tx+(1-t)b)\right\vert dt.
\end{eqnarray*}%
Using the H\"{o}lder's inequality we have;%
\begin{eqnarray*}
&&\left\vert \frac{(b-x)f(b)+(x-a)f(a)}{b-a}-\frac{1}{b-a}%
\dint\limits_{a}^{b}f(u)du\right\vert \\
&\leq &\frac{(x-a)^{2}}{b-a}\left( \dint\limits_{0}^{1}(1-t)^{p}dt\right) ^{%
\frac{1}{p}}\left( \dint\limits_{0}^{1}\left\vert f^{\prime
}(tx+(1-t)a)\right\vert ^{q}dt\right) ^{\frac{1}{q}} \\
&&+\frac{(b-x)^{2}}{b-a}\left( \dint\limits_{0}^{1}(1-t)^{p}dt\right) ^{%
\frac{1}{p}}\left( \dint\limits_{0}^{1}\left\vert f^{\prime
}(tx+(1-t)b)\right\vert ^{q}dt\right) ^{\frac{1}{q}}.
\end{eqnarray*}%
Since $\left\vert f^{\prime }\right\vert ^{q}$ is $m-$convex function;%
\begin{eqnarray*}
&&\left\vert \frac{(b-x)f(b)+(x-a)f(a)}{b-a}-\frac{1}{b-a}%
\dint\limits_{a}^{b}f(u)du\right\vert \\
&\leq &\frac{(x-a)^{2}}{b-a}\left( \frac{1}{p+1}\right) ^{\frac{1}{p}}\left(
\dint\limits_{0}^{1}\left[ t\left\vert f^{\prime }(x)\right\vert
^{q}+m(1-t)\left\vert f^{\prime }(\frac{a}{m})\right\vert ^{q}\right]
dt\right) ^{\frac{1}{q}} \\
&&+\frac{(b-x)^{2}}{b-a}\left( \frac{1}{p+1}\right) ^{\frac{1}{p}}\left(
\dint\limits_{0}^{1}\left[ t\left\vert f^{\prime }(x)\right\vert
^{q}+m(1-t)\left\vert f^{\prime }(\frac{b}{m})\right\vert ^{q}\right]
dt\right) ^{\frac{1}{q}} \\
&=&\left( \frac{1}{p+1}\right) ^{\frac{1}{p}}\left[ \frac{(x-a)^{2}}{b-a}%
\left( \frac{\left\vert f^{\prime }(x)\right\vert ^{q}+m\left\vert f^{\prime
}(\frac{a}{m})\right\vert ^{q}}{2}\right) ^{\frac{1}{q}}\right. \\
&&\left. +\frac{(b-x)^{2}}{b-a}\left( \frac{\left\vert f^{\prime
}(x)\right\vert ^{q}+m\left\vert f^{\prime }(\frac{b}{m})\right\vert ^{q}}{2}%
\right) ^{\frac{1}{q}}\right] .
\end{eqnarray*}%
This completes the proof.
\end{proof}

\begin{corollary}
In Theorem 4, if we choose $x=\frac{a+b}{2},$ we have%
\begin{eqnarray*}
&&\left\vert \frac{f(a)+f(b)}{2}-\frac{1}{b-a}\dint\limits_{a}^{b}f(u)du%
\right\vert \\
&\leq &\left( \frac{1}{p+1}\right) ^{\frac{1}{p}}.\frac{b-a}{4}\left[ \left( 
\frac{\left\vert f^{\prime }\left( \frac{a+b}{2}\right) \right\vert
^{q}+m\left\vert f^{\prime }\left( \frac{a}{m}\right) \right\vert ^{q}}{2}%
\right) ^{\frac{1}{q}}+\left( \frac{\left\vert f^{\prime }\left( \frac{a+b}{2%
}\right) \right\vert ^{q}+m\left\vert f^{\prime }(\frac{b}{m})\right\vert
^{q}}{2}\right) ^{\frac{1}{q}}\right] .
\end{eqnarray*}
\end{corollary}

In Corollary 2, if we choose $m=1$ and $\left( \frac{1}{p+1}\right) ^{\frac{1%
}{p}}\leq 1,$ we obtain%
\begin{equation*}
\left\vert \frac{f(a)+f(b)}{2}-\frac{1}{b-a}\dint\limits_{a}^{b}f(u)du\right%
\vert \leq \frac{b-a}{4}\left[ \left( \frac{\left\vert f^{\prime }\left( 
\frac{a+b}{2}\right) \right\vert ^{q}+\left\vert f^{\prime }\left( a\right)
\right\vert ^{q}}{2}\right) ^{\frac{1}{q}}+\left( \frac{\left\vert f^{\prime
}\left( \frac{a+b}{2}\right) \right\vert ^{q}+\left\vert f^{\prime
}(b)\right\vert ^{q}}{2}\right) ^{\frac{1}{q}}\right] .
\end{equation*}

\begin{theorem}
Let $f:I\subset 
\mathbb{R}
\rightarrow 
\mathbb{R}
$ be a differentiable mapping on $I^{o}$ where $a,b\in I$ with $a<b.$ If $%
\left\vert f^{\prime }\right\vert ^{q}$ is $m-$convex function on $[a,b]$
for some fixed $m\in (0,1],$ then the following inequality holds:%
\begin{eqnarray*}
&&\left\vert \frac{(b-x)f(b)+(x-a)f(a)}{b-a}-\frac{1}{b-a}%
\dint\limits_{a}^{b}f(u)du\right\vert \\
&\leq &\frac{(x-a)^{2}}{2\left( b-a\right) }\left( \frac{\left\vert
f^{\prime }(x)\right\vert ^{q}+2m\left\vert f^{\prime }(\frac{a}{m}%
)\right\vert ^{q}}{3}\right) ^{\frac{1}{q}} \\
&&+\frac{(b-x)^{2}}{2\left( b-a\right) }\left( \frac{\left\vert f^{\prime
}(x)\right\vert ^{q}+2m\left\vert f^{\prime }(\frac{b}{m})\right\vert ^{q}}{3%
}\right) ^{\frac{1}{q}}.
\end{eqnarray*}
\end{theorem}

\begin{proof}
Using Lemma 1, we get;%
\begin{eqnarray*}
&&\left\vert \frac{(b-x)f(b)+(x-a)f(a)}{b-a}-\frac{1}{b-a}%
\dint\limits_{a}^{b}f(u)du\right\vert \\
&\leq &\frac{(x-a)^{2}}{b-a}\dint\limits_{0}^{1}(1-t)\left\vert f^{\prime
}(tx+(1-t)a)\right\vert dt+\frac{(b-x)^{2}}{b-a}\dint\limits_{0}^{1}(1-t)%
\left\vert f^{\prime }(tx+(1-t)b)\right\vert dt.
\end{eqnarray*}%
By the Power mean inequality, we can write;%
\begin{eqnarray*}
&&\left\vert \frac{(b-x)f(b)+(x-a)f(a)}{b-a}-\frac{1}{b-a}%
\dint\limits_{a}^{b}f(u)du\right\vert \\
&\leq &\frac{(x-a)^{2}}{b-a}\left( \dint\limits_{0}^{1}(1-t)dt\right) ^{1-%
\frac{1}{q}}\left( \dint\limits_{0}^{1}(1-t)\left\vert f^{\prime
}(tx+(1-t)a)\right\vert ^{q}dt\right) ^{\frac{1}{q}} \\
&&+\frac{(b-x)^{2}}{b-a}\left( \dint\limits_{0}^{1}(1-t)dt\right) ^{1-\frac{1%
}{q}}\left( \dint\limits_{0}^{1}(1-t)\left\vert f^{\prime
}(tx+(1-t)b)\right\vert ^{q}dt\right) ^{\frac{1}{q}}.
\end{eqnarray*}%
Since $m-$convexity;%
\begin{eqnarray*}
&&\left\vert \frac{(b-x)f(b)+(x-a)f(a)}{b-a}-\frac{1}{b-a}%
\dint\limits_{a}^{b}f(u)du\right\vert \\
&\leq &\frac{(x-a)^{2}}{b-a}\left( \frac{1}{2}\right) ^{1-\frac{1}{q}}\left(
\dint\limits_{0}^{1}(1-t)\left[ t\left\vert f^{\prime }(x)\right\vert
^{q}+m(1-t)\left\vert f^{\prime }(\frac{a}{m})\right\vert ^{q}\right]
dt\right) ^{\frac{1}{q}} \\
&&+\frac{(b-x)^{2}}{b-a}\left( \frac{1}{2}\right) ^{1-\frac{1}{q}}\left(
\dint\limits_{0}^{1}(1-t)\left[ t\left\vert f^{\prime }(x)\right\vert
^{q}+m(1-t)\left\vert f^{\prime }(\frac{b}{m})\right\vert ^{q}\right]
dt\right) ^{\frac{1}{q}} \\
&=&\frac{(x-a)^{2}}{b-a}\left( \frac{1}{2}\right) ^{1-\frac{1}{q}}\left(
\left\vert f^{\prime }(x)\right\vert
^{q}\dint\limits_{0}^{1}(t-t^{2})dt+m\left\vert f^{\prime }(\frac{a}{m}%
)\right\vert ^{q}\dint\limits_{0}^{1}(1-t)^{2}dt\right) ^{\frac{1}{q}} \\
&&+\frac{(b-x)^{2}}{b-a}\left( \frac{1}{2}\right) ^{1-\frac{1}{q}}\left(
\left\vert f^{\prime }(x)\right\vert
^{q}\dint\limits_{0}^{1}(t-t^{2})dt+m\left\vert f^{\prime }(\frac{b}{m}%
)\right\vert ^{q}\dint\limits_{0}^{1}(1-t)^{2}dt\right) ^{\frac{1}{q}} \\
&=&\frac{(x-a)^{2}}{2(b-a)}\left( \frac{\left\vert f^{\prime }(x)\right\vert
^{q}+2m\left\vert f^{\prime }(\frac{a}{m})\right\vert ^{q}}{3}\right) ^{%
\frac{1}{q}} \\
&&+\frac{(b-x)^{2}}{2(b-a)}\left( \frac{\left\vert f^{\prime }(x)\right\vert
^{q}+2m\left\vert f^{\prime }(\frac{b}{m})\right\vert ^{q}}{3}\right) ^{%
\frac{1}{q}}.
\end{eqnarray*}%
This completes the proof.
\end{proof}

\begin{corollary}
In Theorem 5, if we choose $x=\frac{a+b}{2},$ we get%
\begin{eqnarray*}
\left\vert \frac{f(a)+f(b)}{2}-\frac{1}{b-a}\dint\limits_{a}^{b}f(u)du\right%
\vert &\leq &\frac{b-a}{8}\left[ \left( \frac{\left\vert f^{\prime }\left( 
\frac{a+b}{2}\right) \right\vert ^{q}+2m\left\vert f^{\prime }\left( \frac{a%
}{m}\right) \right\vert ^{q}}{3}\right) ^{\frac{1}{q}}\right. \\
&&\left. +\left( \frac{\left\vert f^{\prime }\left( \frac{a+b}{2}\right)
\right\vert ^{q}+2m\left\vert f^{\prime }(\frac{b}{m})\right\vert ^{q}}{3}%
\right) ^{\frac{1}{q}}\right] .
\end{eqnarray*}
\end{corollary}

\begin{remark}
In Corollary 3, if we choose $m=1,$ (\ref{2}) inequality is obtained.
\end{remark}

\section{APPLICATIONS TO SPECIAL MEANS}

We now consider the means for arbitrary real numbers $\alpha ,\beta $ $%
(\alpha \neq \beta ).$ We take

\begin{enumerate}
\item $Arithmetic$ $mean:$%
\begin{equation*}
A(\alpha ,\beta )=\frac{\alpha +\beta }{2},\text{ \ }\alpha ,\beta \in 
\mathbb{R}
^{+}.
\end{equation*}

\item $Logarithmic$ $mean$:%
\begin{equation*}
L(\alpha ,\beta )=\frac{\alpha -\beta }{\ln \left\vert \alpha \right\vert
-\ln \left\vert \beta \right\vert },\text{ \ \ }\left\vert \alpha
\right\vert \neq \left\vert \beta \right\vert ,\text{ }\alpha ,\beta \neq 0,%
\text{ }\alpha ,\beta \in 
\mathbb{R}
^{+}.
\end{equation*}

\item $Generalized$ $log-mean$:%
\begin{equation*}
L_{n}(\alpha ,\beta )=\left[ \frac{\beta ^{n+1}-\alpha ^{n+1}}{(n+1)(\beta
-\alpha )}\right] ^{\frac{1}{n}},\text{ \ \ \ }n\in 
\mathbb{Z}
\backslash \{-1,0\},\text{ }\alpha ,\beta \in 
\mathbb{R}
^{+}.
\end{equation*}
\end{enumerate}

Now using the results of Section 2, we give some applications for special
means of real numbers.

\begin{proposition}
Let $a,b\in 
\mathbb{R}
^{+}$, $a<b,$ $m\in \lbrack 0,1]$ and $n\in 
\mathbb{Z}
,n>1.$ Then, we have%
\begin{equation*}
\left\vert A(a^{n},b^{n})-L_{n}^{n}(a,b)\right\vert \leq n.\frac{b-a}{12}%
\left[ \left\vert \frac{a+b}{2}\right\vert ^{n-1}+m\left\vert \frac{a}{m}%
\right\vert ^{n-1}+m\left\vert \frac{b}{m}\right\vert ^{n-1}\right] .
\end{equation*}%
If we choose $m=1,$ we obtain%
\begin{equation*}
\left\vert A(a^{n},b^{n})-L_{n}^{n}(a,b)\right\vert \leq n.\frac{b-a}{12}%
\left[ \left\vert \frac{a+b}{2}\right\vert ^{n-1}+\left\vert a\right\vert
^{n-1}+\left\vert b\right\vert ^{n-1}\right] .
\end{equation*}
\end{proposition}

\begin{proof}
The assertion follows from Corollary 1 applied to the m-convex mapping $%
f(x)=x^{n},$ $x\in 
\mathbb{R}
,$ $n\in 
\mathbb{Z}
.$
\end{proof}

\begin{proposition}
Let $a,b\in 
\mathbb{R}
^{+}$, $a<b,$ $m\in \lbrack 0,1]$ and $n\in 
\mathbb{Z}
\backslash \{-1,0\}$ $.$ Then, for all $q>1$, we have%
\begin{equation*}
\left\vert A(a^{n},b^{n})-L_{n}^{n}(a,b)\right\vert \leq n.\frac{b-a}{8}%
\left( \left[ \frac{\left\vert \frac{a+b}{2}\right\vert
^{q(n-1)}+2m\left\vert \frac{a}{m}\right\vert ^{q(n-1)}}{3}\right] ^{\frac{1%
}{q}}+\left[ \frac{\left\vert \frac{a+b}{2}\right\vert
^{q(n-1)}+2m\left\vert \frac{b}{m}\right\vert ^{q(n-1)}}{3}\right] ^{\frac{1%
}{q}}\right) .
\end{equation*}%
If we choose $m=1,$ we obtain%
\begin{equation*}
\left\vert A(a^{n},b^{n})-L_{n}^{n}(a,b)\right\vert \leq n.\frac{b-a}{8}%
\left( \left[ \frac{\left\vert \frac{a+b}{2}\right\vert
^{q(n-1)}+2\left\vert a\right\vert ^{q(n-1)}}{3}\right] ^{\frac{1}{q}}+\left[
\frac{\left\vert \frac{a+b}{2}\right\vert ^{q(n-1)}+2\left\vert b\right\vert
^{q(n-1)}}{3}\right] ^{\frac{1}{q}}\right) .
\end{equation*}
\end{proposition}

\begin{proof}
The assertion follows from Corollary 3 applied to the m-convex mapping $%
f(x)=x^{n},$ $x\in 
\mathbb{R}
,$ $n\in 
\mathbb{Z}
.$
\end{proof}

\end{document}